\documentclass[11pt]{article} 

\usepackage[utf8]{inputenc} 

\usepackage{geometry} 
\usepackage{graphicx} 

\usepackage{booktabs} 
\usepackage{array} 
\usepackage{paralist} 
\usepackage{verbatim} 
\usepackage{subfig} 
\usepackage{xcolor}

\usepackage{fancyhdr} 
\usepackage{amsmath,amssymb,amstext,amsthm}
\usepackage{sectsty}

\usepackage[nottoc,notlof,notlot]{tocbibind} 

\makeatother

\newlength{\bibitemsep}
\newlength{\bibparskip}
\let\oldthebibliography\thebibliography
\renewcommand\thebibliography[1]{%
  \oldthebibliography{#1}%
  \setlength{\parskip}{\bibitemsep}%
  \setlength{\itemsep}{\bibparskip}%
}

\usepackage[titles,subfigure]{tocloft} 

\newtheorem{theorem}{\indent Theorem}[section]
\newtheorem{lemma}[theorem]{\indent Lemma}
\newtheorem{remark}[theorem]{\indent Remark}
\newtheorem{proposition}[theorem]{\indent Proposition}
\newtheorem{defn}{Definition}[section]

\newcommand\s{\sigma}

\renewcommand{\l}{\langle}
\renewcommand{\r}{\rangle}
\newcommand{\dive}{{\mathrm{div}\,}}
\newcommand{\dd}{{\mathrm{d}}}
\newcommand{\uu}{{\mathbf{u}}}
\newcommand{\ff}{{\mathbf{f}}}
\newcommand{\vg}{{\mathbf{g}}}

\newcommand{\GG}{{\mathbf{G}}}

\newcommand{\RR}{{\mathbb{R}}}
\newcommand{\e}{\varepsilon}

\newcommand{\p}{\partial}
\newcommand{\aaa}{\alpha}
\newcommand{\calB}{\mathcal{B}}

\newcommand{\uue}{\widetilde \uu_{\e}}

\newcommand\bp{\begin{pmatrix}}
\newcommand\ep{\end{pmatrix}}
\newcommand\be{\begin{equation}}
\newcommand\ee{\end{equation}}
\newcommand\ba{\begin{equation}\begin{aligned}}
\newcommand\ea{\end{aligned}\end{equation}}

\newcommand\nn{\nonumber}

\numberwithin{equation}{section}

\title{Homogenization of Inhomogeneous Incompressible Navier-Stokes Equations in  Domains with Very Tiny Holes}
\author{Yong Lu\footnote{School of Mathematics, Nanjing University, Nanjing 210093, China, luyong@nju.edu.cn} \and Jiaojiao Pan\footnote{School of Mathematics, Nanjing University, Nanjing 210093, China, panjiaojiao@smail.nju.edu.cn}\and Peikang Yang\footnote{School of Mathematics, Nanjing University, Nanjing 210093, China, ypk@smail.nju.edu.cn}}
\date{}

\begin{document}
\maketitle

\noindent
{\bf Abstract} ~In this paper, we study the homogenization problems of  $3D$ inhomogeneous  incompressible Navier-Stokes system perforated with very tiny holes whose diameters are much smaller than their mutual distances.  The key is to establish the equations in the homogeneous domain without holes for the zero extensions of the weak solutions. This allows us to derive time derivative estimates and show the strong convergence of the density and the momentum by Aubin-Lions type argument. For the case of small holes, we finally show the limit equations remain unchanged in the homogenization limit.



\section{Introducton}

In this paper, we study the homogenization of inhomogeneous incompressible Navier-Stokes equations in a perforated domain in $\mathbb{R}^{3}$ with very tiny holes with Dirichlet boundary condition. Our goal is to describe the limit behavior of the (weak) solutions as the number of holes goes to infinity and the size of holes goes to zero simultaneously.

To describe the perforated domain, let $\Omega  \subset {\mathbb{R}^3}$ be a bounded domain of class $C^{2}$. The holes in $\Omega$ are denoted by $\{T_{\e,k}\}_{k\in K_{\e}}$, which are assumed to satisfy
\begin{equation}\label{1-hole}
 T_{\e ,k} =  {x_{\e,k}} + \e^\alpha T_{k}\subset B({x_{\e,k}},{\delta_0}\e^\alpha)   \subset B({x_{\e,k}}, \delta_1\e^\alpha)  \subset B({x_{\e,k}},{\delta_2}\e^\alpha) \subset B({x_{\e,k}},{\delta_3}\e^\alpha) \subset  \Omega,
\end{equation}
where for each $k$, $T_{k}\in \mathbb{R}^3$ is a simply connected bounded domain of class $C^{2}$. Here we assume that $\delta _i, \ i = 0,1,2,3$ are fixed positive numbers independent of $\e$. The perforation parameter $\e^\alpha (\alpha>3)$ is used to measure the size of holes and the minimal mutual distances between the holes are of size $O(\e)$. The perforated domain $\Omega_{\e}$ under consideration is described as:
\begin{equation}\label{1-domain}
\Omega _\varepsilon : = \Omega \backslash \bigcup\limits_{k \in K_\varepsilon} {T_{\varepsilon ,k}}. 
\end{equation}
By the distribution of holes assumed above, the number of holes in $\Omega_{\e}$ satisfy
\be
|K_\varepsilon|\leq C\frac{|\Omega|}{\e^3},\quad \mbox{for some} ~C~ \mbox{independent of} ~ \e.\nn
\ee

 In this paper, we are devoted to  the homogenization of inhomogeneous incompressible Navier-Stokes equations in perforated domain  $\Omega_\varepsilon$  with Dirichlet boundary condition. Our goal is to give a complete description for the homogenization process related to small sizes of holes. Let $T>0$, the initial boundary value problem in space-time cylinder $(0,T)\times\Omega_{\e}$ under consideration is the following:
\begin{equation}\label{equa-inhom-NS}
\begin{cases}
\partial_t \rho_\e + \dive (\rho_\e\uu_\e) = 0, &\mbox{in}\; (0,T)\times\Omega_{\e},\\
\partial_t (\rho_\e\uu_\e) + \dive (\rho_\e\uu_\e \otimes \uu_\e)
-\mu\Delta \uu_\e+ \nabla p_\e = \rho_\e\ff_\e, &\mbox{in}\; (0,T)\times\Omega_{\e},\\
\dive  \uu_\varepsilon = 0, &\mbox{in}\;(0,T)\times\Omega_{\e},\\
\uu_\varepsilon = 0, &\mbox{on}\;(0,T)\times\partial\Omega_{\e},\\
\rho_\e |_{t = 0} = \rho_\e^0,\quad (\rho_\e\uu_\e) |_{t = 0} = \rho_\e^0\uu_\e^0.
\end{cases}
\end{equation}
Here $\rho_\e, \uu_{\e}=(\uu_{\e,1},\uu_{\e,2},\uu_{\e,3})$ are the density and velocity of the fluid respectively, $p_{\e}$ is the fluid pressure. The external force $\ff_{\e}$ is assumed to be in $L^{\infty}((0,T)\times\Omega_{\e};\RR^{3})$, $\mu$ is a positive constant which represents viscosity coefficient. Since the value of $\mu$ is not relevant in our study, we will assume $\mu=1$ in the rest of the paper for simplicity.

The study of homogenization problems in fluid mechanics have gained a lot of interest.  In particular, the homogenization of Stokes system in perforated domains has been systematically studied. In 1980s, Tartar \cite{Tartar} considered the case where the size of holes is proportional to the mutual distance of holes and derived Darcy's law. In 1990s,  Allaire \cite{Allaire1,Allaire2} considered general size $a_\e$ of holes which are periodically distributed and found that the homogenized limit equations are determined by the ratio $\sigma_\e$ given as
\begin{equation}\label{1-sigma}
\sigma _\varepsilon: = \Big(\frac{\varepsilon^d}{a_\varepsilon^{d-2}}\Big)^{\frac{1}{2}},  \ d \geq 3;\quad{\sigma _\varepsilon}: = \varepsilon \left| \log \frac{{a_\varepsilon }}{\varepsilon} \right|^{\frac{1}{2}}, \ d = 2.\nn
\end{equation}
Roughly speaking, if $\lim_{\e\to 0}\s_{\e}=0$ corresponding to the case of supercritical size of holes, the limit behavior is governed by Darcy's law. If $\lim_{\e\to 0}\s_{\e}=+\infty$ corresponding to the case of subcritical size of holes, the limit system coincides with the original system.  If $\lim_{\e\to 0}\s_{\e}=\s_{*}\in (0,+\infty)$ corresponding to critical size of holes, the asymptotic limit behavior is Brinkman's law. Later, Lu \cite{Lu1} gave a unified approach for the homogenization of Stokes equations in perforated domains. Furthermore, Jing, Lu and Prange \cite{Wenjia Jing2} developed a unified method to obtain the quantitative homogenization of Stokes systems in periodically perforated domains. Such quantitative homogenization results were obtained by Jing \cite{Wenjia Jing,Wenjia Jing1}  for Laplace equations and Lamé systems.

As for the homogenization of incompressible Navier-Stokes equations, there are lots of results. Mikeli\'c  \cite{Mikelic} studied the incompressible Navier-Stokes equations in a porous medium, where the size of holes is proportional to the mutual distance of  holes and the Darcy's law is recovered in the limit. Recently, Feireisl, Namlyeyeva and Ne\v{c}asov\'a  \cite{Feireisl-Namlyeyeva-Necasova} studied the case with critical size of holes for the incompressible Navier-Stokes equations and derived Brinkman's law. Lu and Yang \cite{LY-Yang} considered the case of supercritical and subcritical size of holes. At low Mach number, Bella and Oschmann \cite{BO1} considered the critical case for incompressible Navier-Stokes equations, where the Brinkman equations were obtained. H\"{o}fer \cite{Richard M. Hofer} studied the solution to incompressible Navier-Stokes equations in perforated domains in the inviscid limit and proved quantitative convergence results in all regimes.

 The homogenization for the compressible Navier-Stokes equations was considered in \cite{BO2,Diening-Feireisl-Lu,Feireisl-Lu,Lu2} for the case of small holes, in Masmoudi \cite{Masmoudi} for the case of large holes. At low Mach number, H\"{o}fer, Kowalczyk and Schwarzacher \cite{Richard M. Hofer1} studied the case of large holes for the compressible Navier-Stokes equations and derived Darcy's law.  H\"{o}fer, Ne\v{c}asov\'a, Oschmann \cite{Richard M. Hofer2} gave quantitative homogenization of the compressible Navier-Stokes equations towards Darcy’s law.

 The homogenization study is extended to more complicated models describing fluid flows.  Feireisl, Novotn\'y and Takahashi \cite{Feireisl-Novotny1} studied the Navier-Stokes-Fourier equations. Lu and Qian \cite{LY-Qian} considered the homogenization of evolutionary incompressible non-Newtonian flows of Carreau-Yasuda type and Darcy’s law was recovered. Then, Lu and Oschmann \cite{Lu3} gave qualitative/quantitative homogenization of stationary and evolutionary incompressible non-Newtonian flows of Carreau-Yasuda type towards Darcy’s law. 

\medskip

The inhomogeneous incompressible Navier-Stokes equations can be considered as a model for the evolution of a multi-phase flow consisting of several immiscible, incompressible fluids. There are fruitful mathematical results to this model. Global existence results for inhomogeneous, incompressible Navier-Stokes equations were first obtained by Kazhikov\cite{Kazhikov1}-see also Kazhikov and Smagulov \cite{Kazhikov2}. There is also a series of results by Zhang et al. concerning the existence and uniqueness of strong solutions, see \cite{Abidi-Gui-Zhang,Liao-Zhang,Qian-Zhang,Zhang}.  In particular, in Chapter II of the  book \cite{Lions1}, Lions gave a complete description concerning the existence of finite energy weak solutions. The homogenization of $3D$ inhomogeneous incompressible Navier-Stokes system in perforated domains in critical case has been considered by Pan \cite{Pan}, which shows that when $\e\to 0$, the velocity and density converge to a solution of inhomogeneous incompressible Navier-Stokes system with a friction term of Brinkman type. 

In this paper, we are working in the framework of weak solutions in the sense of Lions \cite{Lions1} in perforated spatial domains. Our goal is to show the asymptotic behavior of the solutions when the perforation parameter $\e \to 0$. In the case of small holes, we are able to show that the equations remain unchanged. This coincides with  the previous studies for Stokes equations.  

\subsection{Notations and definition of weak solutions} 

We recall some  notations. Let $W^{1,q}_{0}(\Omega)$ be the collection of Sobolev functions in $W^{1,q}(\Omega)$ with zero trace, and let $W^{-1,q}$ be the dual space of $W^{1,q}_{0}(\Omega)$.  We set $W_{\dive}^{1,q}(\Omega):=\{v\in W_0^{1,q}(\Omega),\ \dive  v =0\}$ with $W^{1,q}$ norm and ${V^{-1,q}}(\Omega)$ be the dual space of ${W_{\dive}^{1,q'}}(\Omega)$. Let $L^{q}_{0}(\Omega)$ be the collection of $L^{q}(\Omega)$ integrable functions with zero integral mean. We sometimes use $L^r L^s$ to denote the Bochner space $L^r(0,T;L^s(\Omega))$ or $L^r(0,T;L^s(\Omega);\RR^{3})$ for short. For a function $g$ defined on $\Omega_{\e}$, we use the notation $\widetilde  g$ to represent its zero extension in $\Omega$:
$$
\widetilde  g = g \ \mbox{in $\Omega_{\e}$}, \quad \widetilde  g = 0 \  \mbox{in $\Omega\setminus \Omega_{\e} = \bigcup\limits_{k \in {K_\varepsilon }} {T_{\varepsilon ,k}}$}.
$$

We  now recall the definition of (finite energy) weak solutions:
\begin{defn}\label{def-weak}
We call $(\rho_\e,\uu_\e)$  a (finite energy) weak solution of \eqref{equa-inhom-NS} in $(0,T)\times\Omega_{\e}$ provided:
\begin{itemize}

\item There holds:
\ba\label{def-weak-1}
& \rho_\e \in L^\infty(0,T;L^\infty(\Omega_\e))\cap C([0,T], L^2(\Omega_\e)),\ \rho_\e\geq 0
\ \mbox{a.e. in} \ (0,T)\times\Omega_{\e},   \\
& \sqrt{\rho_\e} \uu_\e \in L^\infty(0,T;L^2(\Omega_\e)),\  \uu_\e \in L^2(0,T;W^{1,2}_0(\Omega_\e)).\nn
\ea
Moreover
\ba\label{def-weak-2}
&\int_{\Omega_\e} \left| \rho_\e(t,x) \right|^q \dd x
=\int_{\Omega_\e} \left| \rho_\e^0(x) \right|^q \dd x,\, \mbox{for any} \ q\in [1,\infty)
\,\mbox{and}\ t\in [0,T),\\
&\left\| \rho_\e  \right\|_{L^\infty (0,T;L^\infty(\Omega_\e))} = \left\| \rho_\e^0 \right\|_{L^\infty(\Omega_\e)}.
\ea

\item  The continuity equation is satisfied in the weak sense
\ba\label{weak-continuity}
-\int_{\Omega_\e} \rho_\e^0 \psi (0,x) \dd x
=\int_0^T \int_{\Omega_\e} \rho_\e(\partial_t\psi +\uu_\e  \cdot \nabla\psi) \dd x\dd t,\nn
\ea
for any $\psi\in C^\infty_c([0,T)\times\Omega_{\e})$.

\item The momentum equations hold in the sense that
\ba\label{weak-momentum}
& \int_0^T \int_{\Omega_\e}-\rho_\e\uu_\e \cdot \partial_t\varphi - \rho_\e\uu_\varepsilon  \otimes \uu_\varepsilon :\nabla \varphi+\nabla\uu_\e:\nabla \varphi   \dd x\dd t  \\
& = \int_0^T \int_{\Omega_\varepsilon} \rho_\e\ff_\e   \cdot \varphi \dd x\dd t + \int_{\Omega _\e} \rho_\e^0\uu_\e^0  \cdot \varphi (0,x)\dd x,
\ea
for any $\varphi\in C^\infty_c([0,T)\times\Omega_{\e} ;\mathbb{R}^3)$ with $\dive  \varphi =0$.

\item The energy inequality
\ba\label{energy-inequa}
\frac12\int_{\Omega_\e} \rho_\e\left| \uu_\e (t,x) \right|^2 \dd x+\int_0^t \int_{\Omega_\e} \left|\nabla\uu_\e\right|^2 \dd x\dd s
\leq \frac12\int_{\Omega_\e} \rho_\e^0\left| \uu^0_{\e}(x) \right|^2\dd x
+\int_0^t \int_{\Omega_\e} \rho_\e\ff_{\e} \cdot \uu_\e  \dd x\dd s,
\ea
holds for a.e. $t\in  (0,T)$.

\end{itemize}

 Moreover, a finite energy weak solution $(\rho_\e, \uu_\e)$ is said to be a renormalized weak solution if
\ba\label{renorm-equa}
\partial_t \beta(\rho_\e) + \uu_\e\cdot\nabla\beta(\rho_\e) = 0, \ \mbox{in} \ \mathcal{D}'((0,T)\times\RR^3),
\ea
for any $\beta\in C^1([0,\infty))$. In \eqref{renorm-equa}, $(\rho_\e, \uu_\e)$ are extended to be zero outside $\Omega_\e$.
\end{defn}

\begin{remark}\label{remark-1}
For each fixed $\e$, the existence of renormalized weak solutions defined above is known, see for example the systematic study in Lion's book \cite {Lions1}. Furthermore, Lions \cite{Lions1} Theorem  2.3  showed that if $(\rho_\e,\uu_\e)$ is a weak solution of \eqref{equa-inhom-NS} in $(0,T)\times\Omega_{\e}$, then we have $P_{\e}(\rho_\e\uu_\e)\in C_{weak}([0,T];L^2(\Omega_\e))$. Here $P_\e$ is the Leray-Helmholtz projection operator onto the subspace of divergence free vector fields on $\Omega_{\e}$. 
\end{remark}

\subsection{Main Results}

Throughout the paper, we will assume the zero extension of the initial data and the external force satisfy
\ba\label{ini-force}
&\widetilde \rho_{\e}^{0} \to \rho^{0}  &\mbox{strongly in}&  \ L^{\infty}(\Omega),\\
&\widetilde \uu_{\e}^{0}\rightarrow \uu_{0}  &\textup{strongly in}& \ L^{2}(\Omega;\mathbb R^{3}),\\
&\widetilde \ff_{\e} \to \ff  &\mbox{strongly in}&  \ L^{\infty}((0,T)\times\Omega).
\ea
Now we state our results.  Note that the limits are taken up to possible extractions of subsequences. 
\begin{theorem}\label{thm-1}
Let $(\rho_\e,\uu_\e)$ be a renormalized weak solution of the inhomogeneous incompressible Navier-Stokes system \eqref{equa-inhom-NS} in the sense of Definition \ref{def-weak} with initial data and external force satisfying \eqref{ini-force}. Then we have
\ba\label{thm1-conv}
&\widetilde\rho_\e \to \rho \ \mbox{strongly in} \ C([0,T];L^p(\Omega))\ \mbox{for all}\ 1\leq p<\infty,\\
&\uue \to {\uu} \ \mbox{weakly in} \  L^2(0,T;W_0^{1,2}(\Omega)).
\ea
Moreover, $(\rho,\uu)$ is a weak solution of the inhomogeneous incompressible Navier-Stokes equations in the sense of  Definition \ref{def-weak} in homogeneous domain $\Omega$:
\begin{equation}
\begin{cases}\nn 
\partial_t \rho + \dive (\rho\uu) = 0, &\mbox{in}\; (0,T)\times\Omega,\\
\partial_t (\rho\uu) + \dive (\rho\uu \otimes \uu)
-\mu\Delta\uu+ \nabla p = \rho\ff, &\mbox{in}\; (0,T)\times\Omega,\\
\dive  \uu = 0, &\mbox{in}\;(0,T)\times\Omega,\\
\uu = 0, &\mbox{on}\;(0,T)\times\partial\Omega,\\
\rho |_{t = 0} = \rho^0,\quad (\rho\uu) |_{t = 0} = \rho^0\uu^0.
\end{cases}
\end{equation}
\end{theorem}

The rest of the paper is devoted to the proof of above theorem. In the sequel,  $C$ denotes a constant independent of $\varepsilon$, while its value may differ from line to line.

\section{Uniform estimates}
The uniform estimates of $\rho_{\e}$ follow from the continuity equation \eqref{equa-inhom-NS} immediately. By the definition of weak solution $(\rho_\e, \uu_\e)$ in \eqref{def-weak-2} and  the uniform bounds of the initial data in \eqref{ini-force} we have
\ba\label{est-density}
\|\rho_\e \|_{L^\infty(0,T;L^\infty(\Omega_\e))}= \|\rho^0_\e \|_{L^\infty(\Omega_\e)}\leq C.
\ea

The uniform estimates of $\uu_\e$ follow directly from the energy inequality \eqref{energy-inequa}. Indeed, using H\"older's inequality, Poincar\'e inequality and estimate of density \eqref{est-density}, we have
\ba\nn
&\frac12\int_{\Omega_\e} \rho_\e\left| \uu_\e (t,x) \right|^2 \dd x
+\int_0^t \int_{\Omega_\e}  \left|\nabla\uu_\e\right|^2  \dd x\dd s\\
&\leq \int_0^t \int_{\Omega_\e} \rho_\e\ff_{\e} \cdot \uu_\e  \dd x\dd s
+\frac12\int_{\Omega_\e} \rho_e^0\left| \uu^0_{\e}(x) \right|^2\dd x\\
&\leq \left\| \rho_\e^0 \right\|_{L^\infty(\Omega_\e)}
\int_0^t \left(\int_{\Omega_\e} \left| \ff_\e\right|^2   \dd x\right)^{\frac12}
\left(\int_{\Omega_\e} \left| \uu_\e\right|^2  \dd x\right)^{\frac12}\dd s
+\frac12\int_{\Omega_\e} \rho_\e^0\left| \uu^0_{\e}(x) \right|^2\dd x \\
&\leq C\sup_{0<\e\leq1}\left\| \ff_\e \right\|^2_{L^2(0,t;L^2(\Omega_\e))}
+\frac12\left\|  \nabla\uu_\e \right\|^2_{L^2(0,t;L^2(\Omega_\e))}
+\frac12\sup_{0<\e\leq1} \|\rho_\e^0\|_{L^{\infty}(\Omega_\e)}\left\|\uu^0_{\e} \right\|^2_{L^2(\Omega_\e)}.
\ea
Together with the assumptions on the initial data and the external force in \eqref{ini-force}, we deduce
\begin{equation}\nn
\|\sqrt{\rho_\e}\uu_\e \|_{L^\infty(0,T;L^2(\Omega_\varepsilon))} \leq C,
\quad\| \nabla \uu_\e\| _{L^2(0,T;L^2(\Omega_\e))} \leq C.
\end{equation}
Since ${\uu}_\varepsilon \in L^{2}(0,T; W^{1,2}_{0}(\Omega_{\e}))$ has zero trace on the boundary, its zero extension has the estimates:
\begin{equation}\label{est-u}
\|\sqrt{\widetilde\rho_\e}\widetilde \uu_\varepsilon \| _{L^\infty(0,T;L^2(\Omega))} \leq C,
\quad\| \widetilde \uu_\varepsilon \| _{L^2 (0,T;W_0^{1,2}(\Omega))} \leq C.
\end{equation}
Thus, up to a subsequence, there holds the convergence
\ba\label{thm1-conv-0}
{\widetilde {\uu}_\varepsilon } \to {\uu} \ \mbox{weakly in} \  L^2(0,T;W_0^{1,2}(\Omega)),
\ea
which is exactly $\eqref{thm1-conv}_{2}$ in Theorem \ref{thm-1}.
By \eqref{est-density} and \eqref{est-u},  we have
\ba\label{est-rho u}
\| \widetilde\rho_\e\widetilde \uu_\e  \|_{L^2 L^6}
&\leq \|\rho^0_\e \|_{L^\infty} \| \widetilde \uu_\e \|_{L^2 L^6}\leq C,\\
\| \widetilde\rho_\e\widetilde \uu_\e  \|_{L^\infty L^2}&=
\| \sqrt{\widetilde\rho_\e} \sqrt{\widetilde\rho_\e}\widetilde \uu_{\e} \|_{L^\infty L^2}\leq
\|\sqrt{\rho^0_\e} \|_{L^\infty} \| \sqrt{\widetilde\rho_\e}\widetilde \uu_\e\|_{L^\infty L^2}\leq C.
\ea

\section{The continuity equation in homogeneous domain}
In this section, we will derive the limit continuity equation in $(0,T)\times\Omega$. We shall use the following lemma, see for example Lemma 6.8 in \cite{Novotn} and Section 2.2.6 in \cite{Feireisl-Novotny}: \begin{lemma}\label{extend-density}
let $\Omega$ be a bounded Lipschitz domain in $\RR^N, N\geq2$. Let $\rho\in L^2((0,T)\times\Omega), \uu\in L^2(0,T;W_0^{1,2}(\Omega))$ and $f\in L^1((0,T)\times\Omega)$ satisfy
\ba\nn
\partial_t \rho +\dive (\rho\uu) = f,\ \mbox{in}\ \mathcal{D}'((0,T)\times\Omega).
\ea
Then, prolonging $(\rho,\uu,f)$ by $(0,0,0)$ outside $\Omega$,
\ba\nn
\partial_t \rho +\dive (\rho\uu) = f,\ \mbox{in}\ \mathcal{D}'((0,T)\times\RR^N).
\ea
\end{lemma}

Combining the assumptions in Theorem \ref{thm-1} and using Lemma \ref{extend-density} above,  the extension $(\widetilde\rho_\e,\uue)$ satisfies the equation:
\ba\nn
\partial_t \widetilde\rho_\e +\dive (\widetilde\rho_\e\uue) = 0,\ \mbox{in}\ \mathcal{D}'((0,T)\times\RR^3).
\ea

By $\left\| \rho_\e  \right\|_{L^\infty (0,T;L^\infty(\Omega_\e))} = \left\| \rho_\e^0 \right\|_{L^\infty(\Omega_\e)}$ and the initial value conditions  \eqref{ini-force}, we have $0\leq\widetilde\rho_\e\leq C$ a.e. on $(0,T)\times\Omega$.  The convergence of $\widetilde\rho_\e^0$ also holds in $L^p(\Omega)$ for all $1\leq p\leq\infty$. The convergence of
$\uue$ follows from \eqref{thm1-conv-0}. In conclusion, we have properties as follows:
\ba\label{assume-compact}
&0\leq\widetilde\rho_\e\leq C \ \mbox{(a.e.) on}\ (0,T)\times\Omega,\\
&\dive \uue =0\ \mbox{(a.e.) on}\ (0,T)\times\Omega,\quad \|\widetilde \uu_\varepsilon \| _{L^2 (0,T;W_0^{1,2}(\Omega))} \leq C,\\
& \partial_t \widetilde\rho_\e +\dive (\widetilde\rho_\e\uue) = 0,\ \mbox{in}\ \mathcal{D}'((0,T)\times\RR^3 ).\\
&\widetilde\rho_\e^0\to\rho^0\ \mbox{strongly in}\ L^1(\Omega),
\quad\uue\to\uu\ \mbox{weakly in}\ L^2 (0,T;W_0^{1,2}(\Omega)).
\ea
Recall the following compactness result (see Theorem 2.4 in \cite{Lions1}): 
\begin{proposition}\label{compact-density}
If \eqref{assume-compact} holds, then $\widetilde\rho_\e$ converges strongly in $C([0,T];L^p(\Omega))$ for all $1\leq p<\infty$ to the unique solution $\rho$ bounded on $(0,T)\times\Omega$, where
\ba\label{compact-density-1}
\begin{cases}\nn
\partial_t \rho +\dive (\rho\uu) = 0,&\ \mbox{in}\ \mathcal{D}'((0,T)\times\Omega),\\
\rho\in C([0,T];L^1(\Omega)), &\ \rho(0)=\rho^0 \ \mbox{(a.e.) in} \ \Omega.\\
\end{cases}
\ea
\end{proposition}
Using Proposition \ref{compact-density}, we obtain the compactness result
\ba\label{conv-rho}
\widetilde\rho_\e \to \rho \ \mbox{strongly in} \ C([0,T];L^p(\Omega))\ \mbox{for all}\ 1\leq p<\infty,
\ea
which is exactly  $\eqref{thm1-conv}_{1}$  in Theorem \ref{thm-1}, with $\rho\in C([0,T];L^p(\Omega)), \ 1\leq p<\infty, \ \rho(0)=\rho^0 \ \mbox{(a.e.) in} \ \Omega.$
Moreover, Proposition \ref{compact-density} also shows the limit equation of the continuity equation in Theorem \ref{thm-1}, i.e.
\ba\nn
\partial_t \rho +\dive (\rho\uu)= 0.
\ea
\begin{remark}\label{remark-2}  Lions \cite{Lions1} also gave the convergence of $\sqrt{\widetilde\rho_\e}$:
\ba\label{conv-rho-1}
\sqrt{\widetilde\rho_\e} \to \sqrt{\rho} \ \mbox{strongly in} \ C([0,T];L^p(\Omega))\ \mbox{for all}\ 1\leq p<\infty.
\ea
\end{remark}
\section{The momentum equations in homogeneous domain}
To find the limit equations of momentum equations in $\Omega$, an idea is to find a family of functions ${\{ {g_\varepsilon}\} _{\varepsilon>0}}$ that vanish on the holes and converge to $1$ in some Sobolev space $W^{1,q}(\Omega)$, then decompose $\varphi$ as
\be\nn
\varphi  = {g_\varepsilon }\varphi + (1-{g_\varepsilon})\varphi.
\ee
Thus $g_\varepsilon\varphi$ can be treated as a test function for the momentum equations in $\Omega_\varepsilon$. While for the terms related the other part $(1-g_\varepsilon)\varphi$, we show that they are small and converge to zero. However, such a decomposition destroyed the divergence free property of $\varphi$: $\dive (g_{\e} \varphi) \neq 0$. To overcome it, we recall the following Bogovskii type operator in perforated domain $\Omega_{\e}$ (see Proposition 2.2 in \cite{Lu2} and Theorem 2.3 in \cite{Diening-Feireisl-Lu}):
\begin{lemma}\label{lem-div}
Let $\Omega_\e$ be defined through \eqref{1-hole}-\eqref{1-domain} with $a_{\e} = \e^{\alpha}$, $\alpha\geq 1$. Then there exists a linear operator
\ba\nn
\calB_\e : L_0^{q}(\Omega_\e) \to W_0^{1,q}(\Omega_\e; \RR^3),\ 1 < q < \infty,
\ea
such that for any $f\in L_0^{q}(\Omega_\e)$,
\ba\label{pro-div1}
\dive \calB_\e(f) =f \ \mbox{in} \ \Omega_\e,~~\|\calB_\e(f)\|_{W_0^{1,q}(\Omega_\e; \RR^3)}\leq C\, \left(1+\e^{\frac{(3-q)\aaa-3}{q}}\right)\|f\|_{L^q(\Omega_\e)},
\ea
for some constant $C$ independent of $\e$.

 For any $r >3/2$, the linear operator $\mathcal B_\e$ can be extended as a linear operator from $\{\dive\vg: \vg\in L^r(\Omega_\e;\RR^3), \vg \cdot {\bf n}=0 \mbox{ on } \p \Omega_\e\}$ to $L^r(\Omega_\e;\RR^3)$ satisfying
\be\nn
 \|\mathcal{B}_\e(\dive {\bf g})\|_{L^{r}(\Omega_\e;\RR^3)}\leq C \|{\bf g}\|_{L^r(\Omega_\e;\RR^3)},
\ee
for some constant $C$ independent of $\e$.
\end{lemma}

To derive the limit momentum equations in $[0,T)\times\Omega$, we firstly give the following proposition:\begin{proposition}\label{moment-equa}
Under the assumptions in Theorem \ref{thm-1}, the extension $(\widetilde\rho_\e,\widetilde \uu_\varepsilon)$ satisfies the following equations:
\ba\label{weak-momentum1}
& \int_0^T \int_{\Omega}-\widetilde{\rho}_\e\widetilde{\uu}_\e \cdot \partial_t\varphi - \widetilde{\rho}_\e\widetilde{\uu}_\varepsilon  \otimes \widetilde{\uu}_\varepsilon :\nabla \varphi+\mu\nabla\widetilde{\uu}_\e:\nabla \varphi   \dd x\dd t  \\
& = \int_0^T \int_{\Omega} \widetilde{\rho}_\e\widetilde{\ff}_\e   \cdot \varphi \dd x\dd t +\int_{\Omega} \widetilde{\rho}_\e^0\widetilde{\uu}_\e^0  \cdot \varphi (0,x)\dd x+ \l\GG_\varepsilon, \varphi \r,
\ea
where $\varphi\in C^\infty_c([0,T)\times\Omega ;\mathbb{R}^3)$ with $\dive  \varphi =0$. Moreover,
\begin{equation}\label{est-F-3D}
|\l\GG_\varepsilon, \varphi \r| \leq C\varepsilon ^\sigma \big(\|  \partial _t \varphi \| _{L^{\frac 43} L^{2}}
+\| \nabla \varphi\| _{L^{4} L^{r_1}}+\| \varphi (0,x)\| _{L^{r_2}} \big), \ \forall \, \varphi\in C_{c}^{\infty}([0,T)\times\Omega;\mathbb{R}^3), \ \dive \varphi =0.
\end{equation}
Here $\sigma:=((3 - q)\alpha  - 3)/q >0$ for some $q>2$ close to $2$, and $ 2<r_1<3$ is given by ${1}/{r_{1}} = {5}/{6} -{1}/{q}$, $r_2$ is given by $1/2=1/q+1/{r_2}$. 
\end{proposition}
\begin{proof} Let $\varphi  \in C_c^\infty ([0,T)\times\Omega;{\mathbb{R}^3})$ with $\dive\varphi = 0$. By the description of the holes in \eqref{1-hole},  there exist cut-off functions $\{ g_\varepsilon \}_{\varepsilon>0} \subset C^\infty(\RR^{3})$ such that $0\leq g_{\e} \leq 1$ and
\begin{equation}\nn
  g_\varepsilon= 0 \  \mbox{on}  \  \bigcup_{k \in K_\varepsilon} T_{\e,k},\quad g_\varepsilon = 1  \  \mbox{on}  \  (\bigcup_{k \in K_\varepsilon} B(x_{\e,k},\delta_0 \varepsilon^\alpha))^c,\quad |\nabla g_{\e}| \leq C \e^{-\alpha}.
\end{equation}
Then for each $1\leq q \leq \infty$,  there hold
\begin{equation}\label{est-g}
\|  g_\varepsilon- 1\| _{L^q(\RR^{3})} \leq C \varepsilon^{\frac{3\alpha-3}{q}},\quad
\|  \nabla g_\varepsilon\| _{L^q(\RR^{3})} \leq C\varepsilon ^{\frac{3\alpha-3}{q}-\alpha}.
\end{equation}
Now we estimate
\ba\nn
I^\varepsilon &:= \int_0^T \int_\Omega \widetilde\rho_\e\widetilde \uu_\e \partial_t \varphi
+\widetilde\rho_\e\widetilde \uu_\e  \otimes \widetilde \uu_\e :\nabla \varphi
-\mu\nabla \widetilde \uu_\e :\nabla \varphi + \widetilde\rho_\e\widetilde \ff_\e \varphi \dd x \dd t+\int_{\Omega} \widetilde{\rho}_\e^0\widetilde{\uu}_\e^0  \cdot \varphi (0,x)\dd x.
\ea

Firstly, we decompose $\varphi$ into $\varphi  = {g_\varepsilon }\varphi + (1-{g_\varepsilon})\varphi$, then 
\ba\nn
I^\e &=\int_0^T \int_{\Omega_\e} \rho_\e \uu_\e \partial_t (g_\e \varphi)
+ \rho_\e \uu_\e \otimes  \uu_\e:\nabla (g_\e\varphi)-\mu\nabla  \uu_\e:\nabla (g_\e\varphi ) +  \rho_\e\ff_\e (g_\varepsilon \varphi )\dd x \dd t\\
&+\int_{\Omega} \widetilde{\rho}_\e^0\widetilde{\uu}_\e^0  \cdot g_\varepsilon\varphi (0,x)\dd x+\sum_{j=1}^{5} {I_j}
\ea
with
\ba\nn
I_1 & = \int_0^T \int_\Omega \widetilde\rho_\e\widetilde \uu_\e(1-g_\e) \partial_t\varphi \dd x \dd t, \\
I_2 & = \int_0^T \int_\Omega \widetilde\rho_\e\widetilde \uu_\e \otimes \widetilde \uu_\e :(1-g_\e)\nabla \varphi-\widetilde\rho_\e\widetilde\uu_\e \otimes \widetilde \uu_\e:\nabla g_\e\otimes\varphi \dd x\dd t ,\\
I_3 &= -\mu\int_0^T \int_\Omega \nabla \widetilde \uu_\varepsilon :(1-g_\varepsilon)\nabla \varphi
- \nabla \widetilde \uu_\varepsilon:\nabla g_\varepsilon \otimes \varphi \dd x \dd t ,\\
I_4 &= -\int_0^T \int_\Omega \widetilde\rho_\e\widetilde \ff_\e (g_\varepsilon-1)\varphi \dd x\dd t,\\
I_5 &= -\int_{\Omega} \widetilde{\rho}_\e^0\widetilde{\uu}_\e^0  \cdot (g_\varepsilon-1)\varphi (0,x)\dd x.
\ea

Observing that
\ba\label{0-integrals}
\int_{\Omega_{\e}} \varphi  \cdot \nabla g_{\e}\dd x = \int_{\Omega_{\e}} \dive(\varphi g_{\e}) \dd x  = 0,
\ea
 we can apply Lemma \ref{lem-div} and introduce
\ba\nn
\varphi_{1} := g_\e\varphi - \varphi_{2},\quad  \varphi_{2} := \calB_{\e} (\dive(\varphi g_{\e})) = \calB_{\e} (\varphi  \cdot \nabla g_\e).
\ea
Then $\varphi_{1} \in C_{c}^{\infty}([0,T); W^{1,2}_{0}(\Omega_{\e}))$ and satisfy $\dive \varphi_{1} = 0$. Using the weak formulation \eqref{weak-momentum}, we have
\ba\nn
I^\varepsilon & = \int_0^T \int_{\Omega_\e}\rho_\e\uu_\e \partial_t \varphi_{1}
+ \rho_\e\uu_\e \otimes  \uu_\e:\nabla \varphi_{1}-\mu\nabla  \uu_\e:\nabla \varphi_{1} +\rho_\e\ff_\e \cdot  \varphi_{1} \dd x \dd t\\
&~~+\int_{\Omega _\e} \rho_\e^0\uu_\e^0  \cdot \varphi_{1} (0,x)\dd x+\sum_{ j=1}^{5} {I_j} +\sum_{ j=6}^{10} {I_j} \\
&=\sum_{ j=1}^{5} {I_j} +\sum_{ j=6}^{10} {I_j},
\ea
where
\ba
&I_6 =\int_0^T \int_{\Omega_\e}\rho_\e\uu_\e \partial_t \varphi_{2}  \dd x \dd t,\quad &&
I_7 = \int_0^T  \int_{\Omega_\e} \rho_\e\uu_\e \otimes  \uu_\e : \nabla \varphi_{2}   \dd x \dd t , \\
&I_8 =-\mu \int_0^T \int_{\Omega_\e} \nabla  \uu_\e:\nabla \varphi_{2}  \dd x \dd t ,\quad &&
I_9 = \int_0^T \int_{\Omega_{\e}}  \rho_\e\ff_\e \cdot  \varphi_{2}\dd x\dd t,\\
&I_{10} =\int_{\Omega _\e} \rho_\e^0\uu_\e^0  \cdot \varphi_{2} (0,x)\dd x.
\nonumber \ea

Now we estimate $I_{j}$ term by term.  Since $\alpha>3$, there exists $q\in(2, 3)$ close to $2$ such that
\be\nn
\sigma:=\frac{(3 - q)\alpha-3}{q} >0.
\ee
Let $q^*$ be the Sobolev conjugate component to $q$ with $\frac{1}{q^{*}} = \frac{1}{q} - \frac{1}{3}$. Clearly $6<q^{*}<\infty$.

By \eqref{est-rho u} and \eqref{est-g}, together with interpolation inequality  and Sobolev embedding theorem, we have
\ba\nn
|I_1| \leq C \| \widetilde\rho_\e \widetilde \uu_\e \| _{L^4 L^3}\| 1- g_\e \| _{L^6}\| \partial_t\varphi\| _{L^{\frac43} L^2}\leq C \e^\sigma\| \partial _t\varphi \| _{L^{\frac43} L^2} .
\ea

By \eqref{est-u} and \eqref{est-rho u},  using interpolation inequality  and  Sobolev embedding theorem again, we have
\ba
|I_2| \leq C\| \widetilde\rho_\e \widetilde \uu_\e  \otimes \widetilde \uu_\e \| _{L^{\frac{4}{3}} L^2}\big(
\|  g_\e- 1\| _{L^{q^{*}}}\|  \nabla \varphi \| _{L^4L^{r_{1}}} + \| \nabla g_\e\| _{L^q}\| \varphi \| _{L^{4} L^{r_2}}\big) \leq C \e ^\sigma\| \nabla \varphi \| _{L^{4} L^{r_1}} ,
\nonumber\ea
where
\ba\label{r1}
\frac{1}{r_{1}} = \frac{1}{2}  - \frac{1}{q^{*}} = \frac{5}{6} - \frac{1}{q}>\frac{1}{3}, \quad \frac{1}{r_{2}} = \frac{1}{2} - \frac{1}{q}  = \frac{1}{r_{1}^{*}}.
\ea

Similarly, we have 
\ba\nn
|I_3| \leq C\| \nabla \widetilde \uu_\varepsilon \| _{L^2 L^2} \big(\|  g_\varepsilon - 1\| _{L^{q^*}}
\|  \nabla \varphi\| _{L^2 L^{r_1}} +\| \nabla g_\varepsilon\| _{L^q}\|  \varphi\| _{L^2 L^{r_2}}\big)\leq C\varepsilon ^\sigma\| \nabla \varphi\| _{L^2 L^{r_1}},
\ea
\[|I_4| \leq C\|\widetilde\rho^0_\e \|_{L^\infty} \|\widetilde\ff_\e\| _{L^\infty L^\infty}\| g_\e- 1\| _{L^q}\|  \varphi \| _{L^2 L^{r_2}}
\leq C \e ^\sigma\| \nabla\varphi \| _{L^2 L^{r_1}},\]
and
\ba\nn
|I_5| \leq C\|\widetilde\rho^0_\e \|_{L^\infty} \|\widetilde{\uu}_\e^0\| _{L^2}\|g_\e- 1\| _{L^q}\|\varphi(0,x)\| _{L^{r_2}}
\leq C \e ^\sigma\|\varphi(0,x)\| _{L^{r_2}}.
\ea

\medskip

Next we estimate $I_{j}$ $(j=6,7,8,9,10)$, for which the estimates of the Bogovskii operator $\calB_{\e}$ in Lemma \ref{lem-div} will be repeatedly used. Since the Bogovskii operator $\calB_{\e}$ only applies on spatial variable, we have
 \ba\nn
 \p_{t} \varphi_{2} = \p_{t} \calB_{\e}(\varphi  \cdot \nabla  g_{\e}) = \calB_{\e}(\p_{t}\varphi \cdot \nabla  g_{\e}).
 \ea
 Since $q>2$ close to $2$, we can choose $r_{3} > \frac 32$ close to $\frac 32$ and  $r_{4}>1$ close to $1$ such that
$$
 \frac{1}{r_{3}} = \frac{1}{r_{4}} - \frac 13, \quad  \frac{1}{q} + \frac{1}{2} = \frac{1}{r_{4}}.
$$
From \eqref{0-integrals}, we know that $\p_{t}\varphi \cdot \nabla  g_{\e}$ has zero integral mean. Using {\eqref{pro-div1}} and Sobolev embedding, we have
 \ba\nn
|I_6|&  \leq\| \rho_\e\uu_\e \| _{L^4 L^3} \|\p_{t} \varphi_{2} \|_{L^{\frac{4}{3}} L^{r_{3}}}  
\leq C \|\p_{t} \varphi_{2} \|_{L^{\frac{4}{3}} W^{1, r_{4}}} \\
& \leq C \left(1+\e^{\frac{(3-r_{4})\aaa-3}{r_{4}}}\right)\| \p_{t}\varphi \cdot\nabla  g_{\e}\| _{L^{\frac 43} L^{r_{4}}}  \\
& \leq   C\| \nabla  g_{\e}\| _{L^{q}}\|   \p_{t} \varphi \| _{L^{\frac 43} L^{2}}  \leq C \varepsilon^\sigma\|  \p_{t} \varphi \| _{L^{\frac 43} L^{2}},
\ea
 where we used the fact $(3-r_{4})\aaa > \aaa >3$. 

\medskip

 Applying \eqref{0-integrals}, we know that $\varphi \cdot \nabla  g_{\e}$ has zero integral mean. Using \eqref{pro-div1}, similar arguments give
\ba
|I_7|&  \leq\| \rho_\e\uu_\e \otimes  \uu_\e \| _{L^{\frac{4}{3}} L^2} \|\nabla  \calB_{\e}(\varphi \cdot \nabla  g_{\e}) \|_{L^{4} L^{2}}  \leq C \|\varphi \cdot\nabla  g_{\e}\| _{L^{4} L^{2}}  \\
& \leq   C\| \nabla  g_{\e}\| _{L^{q}}\| \varphi \| _{L^{4} L^{r_2}}  \leq C \varepsilon^\sigma\|\nabla\varphi \| _{L^{4} L^{r_1}},\nn
\ea
\ba
|I_8|&  \leq C\|\nabla  \uu_\e \| _{L^{2} L^2} \|\nabla  \calB_{\e}(\varphi \cdot \nabla  g_{\e}) \|_{L^{2} L^{2}}  \leq C \|\varphi \cdot\nabla  g_{\e}\| _{L^{2} L^{2}}  \\
& \leq   C\| \nabla  g_{\e}\| _{L^{q}}\| \varphi \| _{L^{4} L^{r_2}}  \leq C \varepsilon^\sigma\|\nabla\varphi \| _{L^{4} L^{r_1}},\nn
\ea
\ba\nn
|I_9|&  \leq C\|\rho^0_\e \|_{L^\infty} \|\ff_\e\| _{L^\infty L^\infty} \|\calB_{\e}(\varphi \cdot \nabla  g_{\e}) \|_{L^{4} L^{2}}  \leq C \|\varphi \cdot\nabla  g_{\e}\| _{L^{4} L^{2}}  \\
& \leq   C\| \nabla  g_{\e}\| _{L^{q}}\| \varphi \| _{L^{4} L^{r_2}}  \leq C \varepsilon^\sigma\|\nabla\varphi \| _{L^{4} L^{r_1}},
\ea
Thus we have
\ba
 |I_{7}| + |I_{8}|+|I_9| \leq C \varepsilon ^\sigma\| \nabla \varphi \| _{L^{4} L^{r_1}}.\nn
\ea

Furthermore,
\ba\nn
|I_{10}|&\leq C\|\rho^0_\e \|_{L^\infty} \|\uu^0_\e\| _{L^2} \|\calB_{\e}(\varphi(0,x) \cdot \nabla  g_{\e}) \|_{L^{2}}  \leq C \|\varphi(0,x) \cdot\nabla  g_{\e}\| _{L^{2}}  \\
& \leq   C\| \nabla  g_{\e}\| _{L^{q}}\| \varphi(0,x) \|_{L^{r_2}}  \leq C \varepsilon^\sigma\|\varphi (0,x)\| _{L^{r_2}},
\ea

Summing up the above estimates for $I_{j}$ $(j=1,\cdots,10)$, we have
$$
 |I^\varepsilon| \leq C\varepsilon ^\sigma \big(\|  \partial _t \varphi \| _{L^{\frac43} L^{2}}
+\| \nabla \varphi\| _{L^{4} L^{r_1}}+\|\varphi (0,x)\| _{L^{r_2}} \big),
$$
which implies our desired result \eqref{est-F-3D}.
\end{proof}

\section{Convergence of the nonlinear convective term}
Here we will show the convergence of the nonlinear convective term $\widetilde \rho_\e\widetilde \uu_{\e}\otimes \widetilde \uu_{\e}$.  A key observation is that some uniform estimates related to time derivative can be deduced from Proposition \ref{moment-equa}.  Indeed, by Proposition \ref{moment-equa}, for any $\psi\in C_c^\infty ((0,T)\times\Omega ;{\mathbb{R}^3})$ with $\dive\psi=0$, we have
\ba\nn
\left| \l \partial _t (\widetilde \rho_\e\widetilde \uu_\e), \psi \r \right| &
\leq \int_0^T \int_\Omega  \left| \widetilde \rho_\e\widetilde \uu_\e
\otimes \widetilde \uu_\varepsilon :\nabla \psi \right|  + \left| \mu\nabla \widetilde \uu_\varepsilon:\nabla \psi   \right| + \left|  \ff_\e \cdot \psi \right| \dd x\dd t  +| \l\GG_\varepsilon, \psi \r|\\
&\leq C\| \nabla \psi\| _{L^4 L^2} + C\e^\sigma\big(\|  \partial_t \psi \|_{L^{\frac43} L^2}
+\| \nabla \psi\| _{L^{4} L^{r_1}} +\|\psi (0,x)\| _{L^{r_2}}\big) \\
& \leq C\| \nabla \psi\| _{L^{4} L^{r_1}} + C \e^\sigma\| \partial_t \psi \|_{L^{\frac43} L^2},
\ea
where $\sigma>0$ and $r_{1} \in (2,3)$. Let $P$ be the Leray-Helmholtz decomposition on $\Omega$. Thus we have the following decomposition
\ba\label{decomposition-u}
P(\widetilde \rho_\e\widetilde \uu_\e)  =  (\widetilde \rho_\e \widetilde \uu_\e)^{(1)}
+ \e ^\sigma  (\widetilde \rho_\e \widetilde \uu_\e)^{(2)},
\ea
where $\partial_t (\widetilde\rho_\e\widetilde \uu_\e) ^{(1)}$ is uniformly bounded in $L^{\frac{4}{3}}(0,T; W^{-1,r_{1}'}(\Omega)) $ and $(\widetilde \rho_\e\uue)^{(2)}$ is uniformly bounded in $L^4(0,T; L^{2}(\Omega))$, which means 
\be\label{remainder1}
\e ^\sigma(\widetilde \rho_\e\uue)^{(2)}\to 0 \ \mbox{strongly in} \  L^4(0,T; L^{2}(\Omega)).
\ee
Since $\widetilde \rho_\e\uue$ is uniformly bounded in $L^{\infty}(0,T; L^{2}(\Omega)) \cap L^2(0,T; L^6(\Omega))$ (see \eqref{est-rho u}), so is $P(\widetilde \rho_\e\uue)$.  Then $(\widetilde \rho_\e\uue)^{(1)} =P(\widetilde \rho_\e\uue) -  \e^\sigma(\widetilde \rho_\e\uue)^{(2)}$ is uniformly bounded in $L^4(0,T; L^{2}(\Omega))$.

By \eqref{conv-rho},  and the fact that $\uue$ converges weakly to $\uu$ in $L^2(0,T;W_0^{1,2}(\Omega))$, we have
\ba\label{conv-rho-u}
\widetilde\rho_\e \uue \to \rho\uu \ \mbox{in} \ \mathcal{D}'( (0,T)\times\Omega ).
\ea
Using the decomposition \eqref{decomposition-u} and \eqref{remainder1}, together with \eqref{conv-rho-u}, we have
$$
(\widetilde \rho_\e\uue)^{(1)} \to P(\rho\uu) \ \mbox{weakly in} \  L^4(0,T; L^{2}(\Omega)).
$$
By the fact that $\partial_t (\widetilde\rho_\e\widetilde \uu_\e) ^{(1)}$ is uniformly bounded in $L^{\frac{4}{3}}(0,T; W^{-1,r_{1}'}(\Omega)) $ and  $(\widetilde \rho_\e\uue)^{(1)}$ is uniformly bounded in $L^4(0,T; L^{2}(\Omega))$, applying Aubin-Lions type argument (see Lemma 6.4 in \cite{Novotn}) gives
\be\label{Convergence rho-u 1}
(\widetilde \rho_\e\uue) ^{(1)}  \to P(\rho\uu) \ \mbox{strongly in} \ L^{4}(0,T; W^{-1,2}(\Omega)),
\ee
Combining \eqref{decomposition-u}, \eqref{remainder1} and \eqref{Convergence rho-u 1}, we have

\be\label{Convergence P rho-u 1}
P(\widetilde \rho_\e\uue)  \to P(\rho\uu) \ \mbox{strongly in} \ L^{4}(0,T; W^{-1,2}(\Omega)),
\ee
which combined with the fact that $\uue\to \uu$ weakly in $L^{2}(0,T; W^{1,2}_{0}(\Omega))$, implies
\ba\label{conv-ruu1}
P(\widetilde \rho_\e\uue)    \otimes \uue \to  P(\rho\uu)  \otimes \uu \quad \mbox{in} \ \mathcal{D}'((0,T)\times \Omega).
\ea

 Now, we employ  the idea of Lions \cite{Lions1} (Section 2.3)  to show
 \be\nn
 \int_0^T \int_{\Omega} \widetilde\rho_\e| \widetilde\uu_\e |^{2} \dd x\dd t \to \int_0^T \int_{\Omega} \rho| \uu|^{2} \dd x\dd t, \mbox{ as } \e \rightarrow 0.
 \ee
Indeed, since $\dive \widetilde\uu_\e=\dive \uu=0$,  by \eqref{Convergence P rho-u 1},  we have
\ba\label{norm-convergence}
& \int_0^T \int_{\Omega} \widetilde\rho_\e| \widetilde\uu_\e |^{2} \dd x\dd t =  \int_0^T  (\widetilde\rho_\e \widetilde\uu_\e, \widetilde\uu_\e)_{L^{2}(\Omega)}\dd t \\
&=\int_0^T  (P(\widetilde\rho_\e \widetilde\uu_\e), \widetilde\uu_\e)_{L^{2}(\Omega)}\dd t=\int_0^T  \int_{\Omega} P(\widetilde\rho_\e \widetilde\uu_\e)  \cdot \widetilde\uu_\e \, \dd x \, \dd t\\
&{\to}\int_0^T  \int_{\Omega} P(\rho\uu) \cdot \uu \, \dd x \, \dd t=\int_0^T ( P(\rho\uu),\uu) _{L^{2}(\Omega)}\dd t\\
&=\int_0^T (\rho\uu,\uu) _{L^{2}(\Omega)}\dd t =\int_0^T \int_{\Omega} \rho| \uu|^{2} \dd x\dd t, \quad\mbox{as} \quad\e\to 0.
 \ea
Using \eqref{thm1-conv-0} and \eqref{conv-rho} (or \eqref{conv-rho-1}), we have $\sqrt{\widetilde\rho_\e}\widetilde \uu_\varepsilon$ converges to $\sqrt{\rho}\uu$ weakly in $L^{2}((0,T)\times \Omega;\mathbb R^{3})$. Combined with \eqref{norm-convergence}, we immediately have
\be\label{strong convergence 1}
\sqrt{\widetilde\rho_\e}\widetilde \uu_\varepsilon\to \sqrt{\rho}\uu \quad\mbox{strongly in } L^{2}((0,T)\times \Omega;\mathbb R^{3})\mbox{ as } \e \rightarrow 0.
\ee
By  \eqref{est-density} and \eqref{est-u}, $\sqrt{\widetilde\rho_\e}\widetilde \uu_\varepsilon$ is uniformly bounded in  $L^{2}(0,T;L^6(\Omega))$. Combined with $\eqref{strong convergence 1}$, we have
\be\nn
\sqrt{\widetilde\rho_\e}\widetilde \uu_\varepsilon \to \sqrt{\rho}\uu \quad\mbox{strongly in } L^{2}(0,T;L^q(\Omega))\mbox{ as } \e \rightarrow 0,
\ee
for all  $1\leq q<6$. Then we have 
 \ba\label{convective-conv}
\widetilde\rho_\e\uue\otimes \uue  = \sqrt{\widetilde\rho_\e} \widetilde \uu_\varepsilon  \otimes \sqrt{\widetilde\rho_\e}\widetilde \uu_\varepsilon \to \sqrt{\rho}\uu \otimes \sqrt{\rho} \uu =  \rho\uu\otimes \uu \ \mbox{strongly in} \  L^{1}(0,T; L^{s}(\Omega) ),
\ea
for all  $1\leq s<3$.

\section{Passing to the limit}
 Now we are ready to derive the limit equations for momentum equations.

For each $\varphi \in C_{c}^{\infty}([0,T)\times\Omega;\RR^{3}), \ \dive \varphi =0$, combining \eqref{thm1-conv-0} and \eqref{conv-rho}, we have
\ba\label{conv-1}
\int_0^T \int_\Omega (\widetilde\rho_\e \uue) \cdot \partial_{t}\varphi \dd x\dd t \to
\int_0^T \int_\Omega  (\rho\uu) \cdot \partial_{t}\varphi \dd x\dd t, \mbox{ as } \e \rightarrow 0.
\ea
By \eqref{convective-conv}, we immediately have 
\ba\label{conv 2}
\int_0^T \int_{\Omega} \widetilde{\rho}_\e\widetilde{\uu}_\varepsilon  \otimes \widetilde{\uu}_\varepsilon :\nabla \varphi \dd x\dd t \to\int_0^T \int_{\Omega}\rho\uu  \otimes\uu:\nabla \varphi \dd x\dd t, \mbox{ as } \e \rightarrow 0.
\ea
Then the weak convergence of $\uue$ in \eqref{thm1-conv-0} implies
\ba\label{conv 3}
\mu\int_0^T \int_{\Omega} \nabla\widetilde{\uu}_\e:\nabla \varphi \dd x\dd t \to\mu\int_0^T \int_{\Omega} \nabla\uu:\nabla \varphi \dd x\dd t, \mbox{ as } \e \rightarrow 0.
\ea
Using the strong convergence of $\widetilde  \rho_{\e}$ in \eqref{conv-rho} and $\widetilde{\ff}_\e$ in \eqref{ini-force}, we have
\ba\label{conv 4}
\int_0^T \int_{\Omega} \widetilde{\rho}_\e\widetilde{\ff}_\e   \cdot \varphi \dd x\dd t \to\int_0^T \int_{\Omega} \rho\ff  \cdot \varphi \dd x\dd t, \mbox{ as } \e \rightarrow 0.
\ea
Since  $\widetilde\rho_{\e}^{0}\uue^0 \to \rho^0\uu^0 \  \mbox{strongly in}  \ L^2(\Omega)$ by the initial value conditions \eqref{ini-force}, we have
\ba\label{conv-5}
\int_\Omega (\widetilde\rho_\e^0\widetilde \uu_\e^0)  \cdot \varphi(0,x)   \,\dd x \to \int_\Omega  (\rho^0\uu^0) \cdot \varphi(0,x) \dd x, \mbox{ as } \e \rightarrow 0.
\ea
Moreover, \eqref{est-F-3D} shows
\ba\label{limit-GG}
\l\GG_\varepsilon, \varphi \r\to 0, \mbox{ as } \e \rightarrow 0.
\ea
Combining \eqref{conv-1}-\eqref{limit-GG}, passing $\e\to 0$ in \eqref{weak-momentum1}, we deduce that for each $\varphi \in C_{c}^{\infty}([0,T)\times\Omega;\RR^{3}), \ \dive \varphi =0$,
\ba\nn 
\int_{0}^{T} \int_\Omega - (\rho\uu) \p_{t} \varphi - (\rho\uu)\otimes \uu : \nabla \varphi  + \mu\nabla \uu :\nabla \varphi \dd x \dd t  = \int_{0}^{T} \int_\Omega \rho\ff \cdot \varphi\dd x\dd t+ \int_\Omega  (\rho^0\uu^0) \cdot \varphi(0,x) \dd x.
\ea
Thus we complete the proof of Theorem \ref{thm-1}.
\begin{remark}\label{remark-3}
 Lions \cite{Lions1} Theorem  2.3  shows that if $(\rho,\uu)$ is a weak solution of \eqref{equa-inhom-NS} in $(0,T)\times\Omega$, then we have $P(\rho\uu)\in C_{weak}([0,T];L^2(\Omega))$. Here $P$ is the Leray-Helmholtz projection operator onto the subspace of divergence free vector fields on $\Omega$. 
 \end{remark}

\section*{Acknowledgements}
{\it Y. Lu and J. Pan are partially supported  by the Natural Science Foundation of Jiangsu Province under grant BK20240058 and by the National Natural Science Foundation of China under Grant 12171235.}

\section*{Conflict of interest}
The authors declare no conflict of interest in this paper.

\end{document}